\documentclass[11pt]{amsart}
\usepackage{amsmath,amssymb,amsthm}
\usepackage[latin1]{inputenc}
\usepackage{version,tabularx,multicol}
\usepackage{graphicx,float, url}
\usepackage{amsfonts}

\usepackage{graphicx}
\usepackage{epsfig}
\usepackage{subfigure}
\usepackage{float}
\usepackage{color}
\usepackage{flafter}

\vspace{0.3cm}

\theoremstyle{plain}
\newtheorem{theo}{Theorem}[section]

\newtheorem{cor}[theo]{Corollary}
\newtheorem{prop}[theo]{Proposition}

\newtheorem{lem}[theo]{Lemma}

\theoremstyle{remark}

\newcommand{\cd}{{\mathcal D}}

\newcommand{\cp}{{\mathcal P}}

\newcommand{\E}{{\mathbb E}}

\newcommand{\N}{{\mathbb N}}
\renewcommand{\P}{{\mathbb P}}

\newcommand{\xen}[1]{X^{(n)}_{#1}}
\newcommand{\yyen}[1]{Y^{(n)}_{#1}}

\newcommand{\zen}[1]{Z^{(n)}_{#1}}
\newcommand{\uen}[1]{U^{(n)}_{#1}}

\newcommand{\hen}[1]{\hat I^{(n)}_{#1}}
\newcommand{\thtn}[1]{\theta^{(n)}_{#1}}
\newcommand{\etang}[1]{\eta^{(n)}_{#1}}
\newcommand{\thtng}{\theta^{(n)}_\gamma}
\newcommand{\ien}{I^{(n)}}
\newcommand{\een}{E^{(n)}}
\newcommand{\len}{L^{(n)}}
\newcommand{\tien}{\tilde I^{(n)}}
\newcommand{\hien}{\hat I^{(n)}}
\newcommand{\ton}{\tau^{(n)}}
\newcommand{\son}{\sigma^{(n)}}

\newcommand{\etan}{\eta^{(n)}}
\newcommand{\hetan}{\hat\eta^{(n)}}
\newcommand{\zetan}{\zeta^{(n)}}

\newcommand{\men}[1]{M^{(n)}_{#1}}

\begin{document}

\title{Total internal and external lengths of the Bolthausen-Sznitman coalescent}

\date{\today}

 \author{G\"otz Kersting}
  \address{Goethe Universit\"at, Robert Mayer Strasse 10, D-60325 Frankfurt am Main, Germany}
 \email{kersting@math.uni-frankfurt.de}

\author{Juan Carlos Pardo}
\address{CIMAT, A.C., Calle Jalisco s/n,
Col. Mineral de Valenciana,
36240 Guanajuato, Guanajuato, Mexico}
\email{jcpardo@cimat.mx}

\author{Arno Siri-J{\'e}gousse}
\address{CIMAT, A.C., Calle Jalisco s/n,
Col. Mineral de Valenciana,
36240 Guanajuato, Guanajuato, Mexico}
\email{arno@cimat.mx}

\begin{abstract} In this paper, we study a weak law of large numbers for  the total internal length of the Bolthausen-Szmitman coalescent.  As a consequence, we obtain the weak limit law of the centered and rescaled total external length. The latter extends results obtained by Dhersin $\&$ M\"ohle \cite{DM12}. 
An application to population genetics dealing with the total number of mutations in the genealogical tree is also given.
 \end{abstract}


\maketitle

\section{Introduction and main results}

In population genetics, one way to explain disparity is to observe how many genes appear only once in the sample.
A gene carried by a single individual is the result of two possible events:
either the gene comes from a mutation that appeared in an external branch of the genealogical tree, either this gene is of the ancestral type and mutations occured in the rest of the sample
(see Figure \ref{fig:graph1}).
We suppose that events of the second type occur in a much less frequent way than events of the first type (it is indeed the case when the size of the sample goes big). 
 The total number of genes carried by a single individual is then closely related to the so-called total external length, which is the sum of all external branch lengths of the tree.

\begin{figure}[h]
\begin{center}
\includegraphics[scale=0.5]{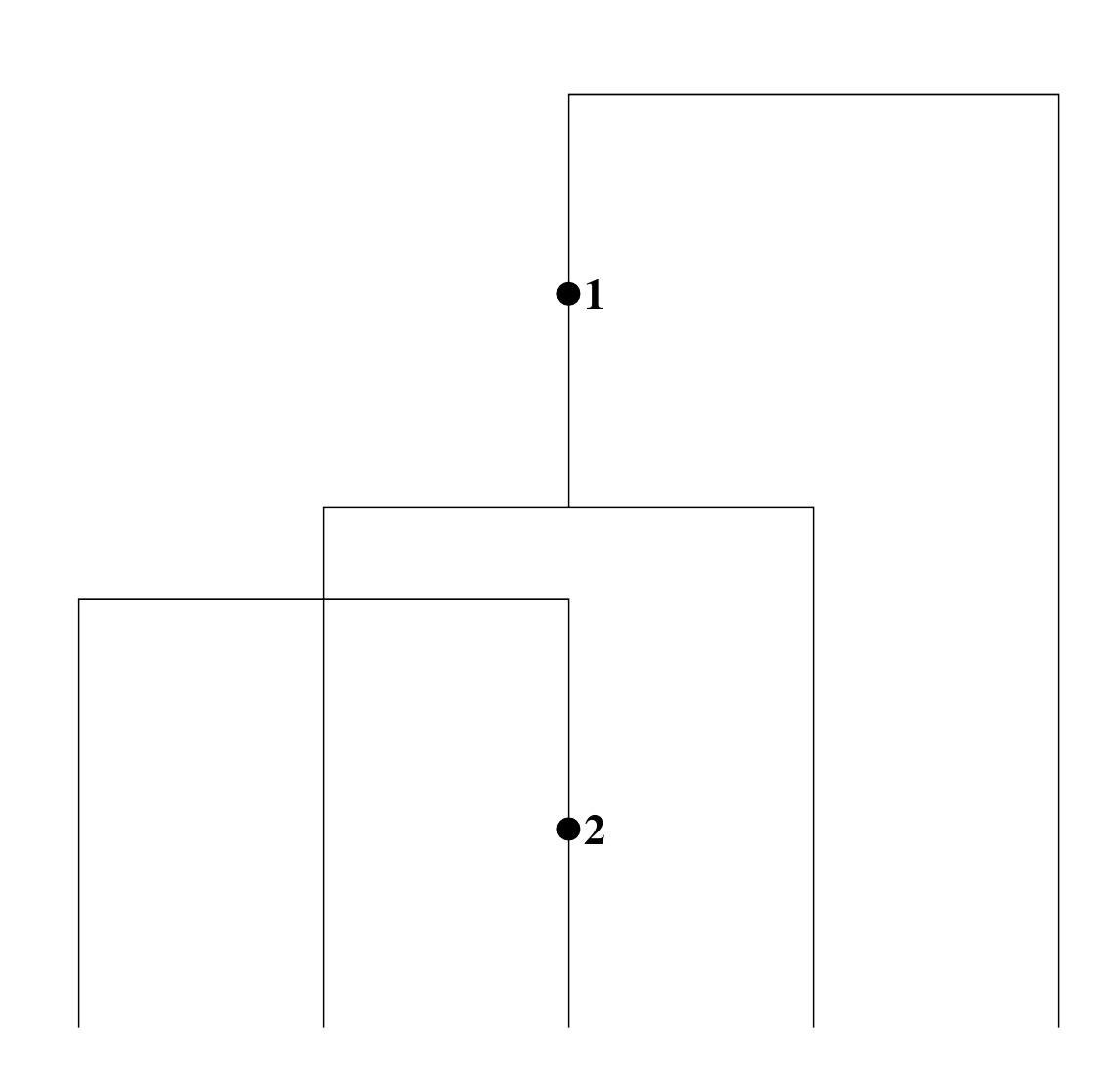}
\caption{In this genealogical tree, two mutations appear. Mutation 1 is in an internal branch and it is shared by 4 individuals.
Mutation 2 is in an external branch so it is carried by 1 individual.
In this example, an ancestral gene is also carried by 1 individual. This situation is  negligible when the size of the sample is large.} \label{fig:graph1}
\end{center}
\end{figure}

 The Bolthausen-Sznitman coalescent (see for instance \cite{BS98}) is a well-known example of exchangeable coalescents with multiple collisions (see \cite{Pit99, Sag99} for a proper definition of this type of coalescents). It was first introduced in physics, in order to study  spin glasses but it has also been thought as a limiting genealogical model for evolving populations with selective killing at each generation, see for instance \cite{BDMM06, BDMM07}. Recently, Berestycki et al. in \cite{BBS10} noted that this coalescent represents the genealogies of the branching brownian motion with absorption.

The Bolthausen-Sznitman coalescent $(\Pi_t,t\geq0)$, is a continuous time Markov chain with values in the set of partitions of $\N$, starting with an infinite number of blocks/individuals. In order to give a formal description of the Bolthausen-Sznitman coalescent, it is sufficient to give its jump rates. Let $n\in \N$, then the restriction $(\Pi_t^{(n)},t\geq0)$ of $(\Pi_t,t\geq0)$ to $[n]=\{1,\dots,n\}$ is a Markov chain with values in $\cp_n$, the set of partitions of $[n]$, with the following dynamics: whenever $\Pi_t^{(n)}$ is a partition consisting of  $b$ blocks, any particular $k$ of them merge into one block at rate 
\[
\lambda_{b,k}=\frac{(k-2)!(b-k)!}{(b-1)!},
\]
so the next coalescence event occurs at rate
\begin{equation}\label{eq:totalrate}
 \lambda_{b}=\sum_{k=2}^b\binom{b}{k}\lambda_{b,k}=b-1.
\end{equation}
Note that mergers of several blocks into a single block is possible, but multiple mergers do not occur simultaneously. Moreover, this 
coalescent process is exchangeable, i.e.  its law does not change under the effect of a random permutation of labels of its blocks.

One of our aims is to study the total external length of the Bolthausen-Sznitman coalescent. More precisely, we  determine the asymptotic behaviour of the total external length $E^{(n)}$ of the Bolthausen-Sznitman coalescent restricted to $\cp_n$, when $n$ goes to infinity, and relate it to its total length $L^{(n)}$ (the sum of lengths of all external and internal branches). A first orientation can be gained from coalescents without proper frequencies. For this class M\"ohle \cite{Moh10} proved that after a suitable scaling the asymptotic distributions of $E^{(n)}$ and $L^{(n)}$ are equal. Now the Bolthausen-Sznitman coalescent does not belong to this class, but it is (loosely speaking) located at the borderline. Also it is known for the Bolthausen-Sznitman coalescent  \cite{DIMR07} that
\begin{equation} \label{eqlen}
\frac{(\log n)^2}{n} L^{(n)}-\log n-\log\log n\xrightarrow[n\to\infty]{d}Z,
\end{equation}
where $\xrightarrow[]{d}$ denotes convergence in distribution and $Z$ is a strictly stable r.v. with index 1, i.e. its  characteristic exponent  satisfies
\[
\Psi(\theta)=-\log\mathbb{E}\Big[e^{i\theta Z}\Big]=\frac{\pi}{2}|\theta|-i\theta\log|\theta|, \qquad \theta \in \mathbb{R}.
\]
In their recent work, Dhersin and M\"ohle
\cite{DM12} showed that the ratio $E^{(n)}/L^{(n)}$ converges to 1 in probability.
Thus one might guess that $E^{(n)}$ satisfies the same asymptotic relation with the same scaling. It is a main result of this paper that such a conjecture is almost, but not completely true.\\


Let us consider $(\Pi^{(n)}_t,t\geq0)$.  We denote by $\uen k$ the size of the $k$-th jump, i.e the number of blocks that the Markov chain loses in $k$-th coalescence event.
We also denote by  $\xen{k}$ for  the number of blocks after $k$ coalescence events. 
Observe that $\xen{0}=n$ and 
$\xen{k}=\xen{k-1}-\uen k=n-\sum_{i=1}^k\uen i$.
Since the merging blocks coalesce into 1, there are $\uen k+1$ blocks involved in $k$-th coalescence event and, for $l< \xen{k-1}$,
$$
\P\Big(\uen k=l\Big|\xen{k-1}=b\Big)=\frac{\displaystyle\binom{b}{l+1}\lambda_{b,l+1}}{\lambda_{b}}=\frac b{b-1}\frac 1{l(l+1)}.
$$
Let $\ton$ be the number of coalescence events. More precisely
$$
\ton=\inf\Big\{k, \xen{k}=1\Big\}.
$$
According to Iksanov and M\"ohle \cite{IM07} (see also \cite{DIMR09}),  $\ton$ satisfies the following asymptotic behaviour
\begin{equation}\label{eq:abton}
\frac{(\log n)^2}{n}\ton-\log n-\log\log n
\xrightarrow[n\to\infty]{d} Z.
\end{equation}

The main result of this paper describes the behaviour of  the total internal length  $\ien$, when $n$ goes to $\infty$. In order to do so, we introduce the r.v. $\yyen k$ that represents the number of internal branches after $k$ coalescence events. In other words, it is 
 the number of remaining blocks which have already participated in a coalescence event. Note that at time 0 all branches are external i.e.  $\yyen 0=0$. Let  $(\mathbf{e}_k,k\geq1)$ be a sequence of i.i.d. standard exponential r.v., also independent from the $\xen k$ and $\yyen k$, thus  from \eqref{eq:totalrate}, we have
$$\ien =\sum_{k=1}^{\ton-1}{\yyen k}\frac{\mathbf{e}_k}{\xen k-1}.
$$
Our main result is the following weak law of large numbers for $\ien$. Here $\xrightarrow[]{\mathbb{P}} $ denotes convergence in probability. 
\begin{theo}\label{th:cvI}
For the total internal length of the Bolthausen-Sznitman coalescent, we have
$$
\frac{(\log n)^2}{n}\ien \xrightarrow[n\to\infty]{\mathbb{P}} 1.
$$
\end{theo}

Now noting that  $\len=\ien+\een$ and using  \eqref{eqlen} and our main result, we deduce  the asymptotic distribution  of the total external length $\een$.
\begin{cor}\label{cor:cvE}
For the total external length of the Bolthausen-Sznitman coalescent, we have
$$
\frac{(\log n)^2}{n}\een-\log n-\log\log n\xrightarrow[n\to\infty]{d}Z-1.
$$

\end{cor}

Observe that the Bolthausen-Sznitman coalescent can be seen as a special case ($\alpha=1$) of the so-called $Beta(2-\alpha,\alpha)$-coalescent which class is defined for $0<\alpha<2$. 
M\"ohle's work \cite{Moh10} shows that in the case $0<\alpha<1$ the variable $\een/n$ converges in law to a random variable defined in terms of a driftless subordinator depending on $\alpha$.
For $1<\alpha<2$, we refer to \cite{KSW13} where it is proven that
$(\een-cn^{2-\alpha})/n^{1/\alpha+1-\alpha}$ converges weakly to a stable r.v. of index $\alpha$, $c$ being a constant also depending on $\alpha$ (see also \cite{BBS08, BBS07, DY12}).
In Kingman's case ($\alpha\to2$) a logarithmic correction appears and the limit law is normal (see \cite{JK11}).

The remainder of the paper is structured as follows. In Section 2, we prove our main results using a  coupling method which was introduced in \cite{IM07} that provides more information of the chain $\xen{}=(\xen k,k\geq0)$. Finally,  Section \ref{sec:mut} is devoted to the asymptotic behaviour  of the number of mutations appearing in external and internal branches of the Bolthausen-Sznitman coalescent.

\section{Proofs}
\subsection{A coupling.} 

In this section, we use the coupling method introduced in \cite{IM07} in order to study the number of jumps $\ton$.

Let $(V_i)_{i\ge 1}$ be a sequence of  i.i.d. random variables with distribution  
\begin{equation}\label{eq:tailV}
\P(V_1=k)=\frac{1}{k(k+1)}, \qquad k\geq1.
\end{equation}
Note that $\P(V_1\geq k)=1/k.$ Let $S_n=V_1+\dots+V_n$.
It is well-known, see for instance \cite{GdH00}, that 
\begin{equation}\label{GdH}
\frac{S_n-n\log n}{n}
\xrightarrow[n\to\infty]{d}Z,
\end{equation}
where   $Z$ is the  stable random variable  that appears in \eqref{eqlen}.
We have the following functional limit result,
with a limit, which is certainly a L\'evy process.

\begin{lem}\label{lem:fltV}
The process $(L_n(t),0\leq t\leq1)$ defined by
 $$
L_n(t)=\frac{S_{\lfloor nt\rfloor}-\lfloor nt\rfloor\log n}{n}
$$
converges weakly in the Skorohod space $\cd[0,1]$.
\end{lem}

\begin{proof} We first verify that the convergence of finite-dimensional distributions holds. Let $t\ge 0$, from (\ref{GdH}), we deduce
 $$
 L_n(t)=\frac{S_{\lfloor nt\rfloor}-\lfloor nt\rfloor\log (nt)+\lfloor nt\rfloor\log t}{n}\xrightarrow[n\to\infty]{d}tZ+t\log t.
$$
Similarly, if we take $s\le t$, then 
\[
\Big(L_n(s),L_n(t)-L_n(s)\Big)\xrightarrow[n\to\infty]{d}(Z_1,Z_2),
\]
where $Z_1$ and $Z_2$ are independent random variables distributed  as $sZ+s\log s$ and $(t-s)Z+(t-s)\log(t-s)$, respectively.  The  mapping theorem implies that 
\[
\Big(L_n(s),L_n(t)\Big)\xrightarrow[n\to\infty]{d}(Z_1,Z_1+Z_2).
\]
A set of three or more time points can be treated in the same way, and hence the finite-dimensional distributions converge properly.

We  now check tightness via Aldous' criterion. Let $T_n$ be a  $L_n$-stopping time and $(\theta_n)$ a sequence of positive numbers such that
$\theta_n\to 0$, as $n$ increases. Then for $\varepsilon>0$, we have
\begin{align*}
 \P\Big(\Big| L_n(T_n+\theta_n)- L_n(T_n)\Big|=\varepsilon\Big)
& \leq \P\Big(\Big| L_n(\theta_n)\Big|\geq\varepsilon\Big)\\
&=\P\left(\left| \frac{S_{\lfloor n\theta_n\rfloor}-\lfloor n\theta_n\rfloor\log (n\theta_n)}{n\theta_n}\theta_n+\frac{\lfloor n\theta_n\rfloor\log \theta_n}{n}\right|\geq\varepsilon\right)
\end{align*}
which converges to 0. This completes the proof.
\end{proof}

In what follows, we use the following notation. 
For a stochastic process $(Z_n,n\ge 0)$ and a function $c(n)$ write $Z_n=O_p(c(n))$ as $n\to \infty$, if   $Z_n/c(n)$ is stochastically bounded as $n\to\infty$, i.e. if $\lim_{x\to\infty}\limsup_{n\to\infty}\P(|Z_n|>xc(n))=0$. 
We also write $Z_n=o_p(c(n))$ as $n\to\infty$, if $Z_n/c(n)$ goes to $0$ in probability.\\

From the above result, we deduce 
\begin{equation}\label{eq:supV}
\sup_{1\leq k\leq n}\Big|S_k-k\log n\Big|=O_p(n).
\end{equation}
We now define recursively  $(\rho(k))_{k\ge 0}$, a sequence of stopping times such that $ \rho(0)=0$ and 
$$
\rho(k+1)=\inf\left\{i>\rho(k), V_i+\sum_{j=1}^kV_{\rho(j)}< n\right\}
$$
with the convention $\inf\{\emptyset\}=\infty$.
In other words, the sequence $(\rho(k))_{k\ge 1}$ is the collection of indices of  the r.v.'s  $V_i$ such that their sum does not exceed $n-1$. It is proved in \cite{IM07} that $\ton$ and $\sup\{k,\rho(k)<\infty\}$ are equal in law, and that the terms of the block-counting Markov chain of the Bolthausen-Sznitman coalescent can be represented as $\xen 0=n$, and
$$
\xen k=n-\sum_{i=1}^kV_{\rho(i)}.
$$
Next, we define
$$
\son=\inf\{k,\rho(k)>k\},
$$
the first time that the random walk meets or exceeds $n$, and 
\begin{equation}\label{eq:defth}
\thtn\gamma=\ton-\frac{n}{(\log n)^{1+\gamma}}, \qquad \gamma \in (0,\infty],
\end{equation}
with the convention $\thtn\infty=\ton$.
Our first result allows  us to consider the random walk instead of the process of disappearing blocks  until time $\thtng$. 

\begin{prop}\label{lem:grolem}
 Let $0<\gamma<\gamma'\leq\infty$. Then as $n$ goes to $\infty$, we have
 \begin{equation}\label{eq:aser2}
 \P\Big(\thtng<\son\Big)\to 1\qquad \textrm{ and } \qquad  \frac{(\log n)^{\gamma}}{n}\xen\thtng
 \xrightarrow[]{\mathbb{P}}1.
\end{equation}
Moreover, 
\begin{equation}\label{eq:aser3}
\sup_{1\leq k\leq \thtng}\left|
\frac{\xen k}{\thtn{\gamma'}-k}-\log n\right|=o_p(\log n),
\end{equation}
for $n$ sufficiently large. 
\end{prop}
In order to prove this proposition, we first show that a similar result holds for the family  of stopping times
$$
\etang{c,\gamma}
=\inf\left\{k,\xen k<\frac{cn}{(\log n)^{\gamma}}\right\},
$$
where $c$ is a positive constant, and then note that for $\epsilon>0$,
\[
\P\Big(\etang{1-\epsilon, \gamma}\le \thtn\gamma\le\etang{1+\epsilon, \gamma} \Big)\xrightarrow[n\to\infty]{}1.
\]
Hence,  the proof of Proposition \ref{lem:grolem} relies on the following Lemma.
\begin{lem}\label{lem:grolem1}
 Let $0<\gamma<\infty$. Then as $n$ goes to $\infty$, we have
 \begin{equation*}\label{eq:aserLem1}
 \P\Big(\etang{c,\gamma}<\son\Big)\to 1\qquad \textrm{ and } \qquad  \frac{(\log n)^{\gamma}}{cn}\xen{\etang{c,\gamma}}
 \xrightarrow[]{\mathbb{P}}1.
\end{equation*}
Moreover, 
\begin{equation*}\label{eq:aserLem2}
\sup_{1\leq k\leq \etang{c,\gamma}}\left|
\frac{\xen k}{\tau^{(n)}-k}-\log n\right|=o_p(\log n),
\end{equation*}
for $n$ sufficiently large. 
\end{lem}
\begin{proof} First we prove the result for $\gamma<1$. Observe, from 
  \eqref{eq:tailV},  that for any $\varepsilon>0$
 $$
\P\left(V_k<\frac{n}{(\log n)^{1-\varepsilon}}, \text{ for all }k\leq2n/\log n\right)
\ge\left(1-\frac{(\log n)^{1-\varepsilon}}{n}\right)^{2n/\log n}\xrightarrow[n\to\infty]{} 1.
$$
Now, since $\P(\ton\leq\frac{2n}{\log n})
\to 1$, as $n$ increases (which follows from \eqref{eq:abton}), we get 
\begin{equation}\label{eq:maxV}
 \sup_{1\leq k\leq\ton}V_k=O_p\left(\frac{n}{(\log n)^{1-\varepsilon}}\right).
\end{equation}
For simplicity, we write $\etan$ instead of $\etang{c,\gamma}$. Then, it follows
\begin{align*}
 \P\Big(\son\leq\etan\Big)
 &=\P\Big(V_k\geq\xen{k-1}, \,\,\text{ for some }k\leq\etan\Big)\\
 &\leq\P\left(V_k\geq\frac{cn}{(\log n)^{\gamma}},\,\, \text{ for some }k\leq\ton\right)\\
&=\P\left(\sup_{1\leq k\leq\ton}V_k\geq\frac{cn}{(\log n)^{\gamma}}\right),
\end{align*}
thus if we take $\varepsilon\in(0,1-\gamma)$ in (\ref{eq:maxV}), we deduce
\begin{equation}\label{eq:siginfeta}
 \P\Big(\son\leq\etan\Big)
\xrightarrow[n\to\infty]{} 0.
\end{equation}
On the event $\{\son>\etan\}$, it is clear that
$$
 \sup_{1\leq k\leq\etan}\left|\frac{\xen k}{\xen {k-1}}-1\right|
 =\sup_{1\leq k\leq\etan}\frac{V_k }{\xen {k-1}}
\leq \frac{(\log n)^\gamma}{cn}\sup_{1\leq k\leq\ton}V_k.
$$
Hence,   from (\ref{eq:maxV}) and  (\ref{eq:siginfeta}), we obtain
\begin{equation}
 \label{eq:ratio succ}
  \sup_{1\leq k\leq\etan}\left|\frac{\xen k}{\xen {k-1}}-1\right|
=o_p(1).
\end{equation}
In particular, since 
$\xen\etan\leq\frac{cn}{(\log n)^\gamma}\leq\xen{\etan-1}$, we get 
\begin{equation}
 \label{eq:asser2step1}
\frac{(\log n)^\gamma}{cn}\xen\etan
\xrightarrow[n\to\infty]{\P}1.
\end{equation}
Next, we note
$$ \xen k-(\ton-k)\log n
 =\xen k-n+k\log\left(\frac{2n}{\log n}\right)
 +(n-\ton\log n)+k\log\left(\frac{\log n}{2}\right),
$$
and from (\ref{eq:abton}), it is clear
\[
\frac{(\log n)^2}{n}\ton=\log n +O_p(\log\log n).
\]
 Then on the event
$\{\etan<\son,\etan<\frac{2n}{\log n}\}$, it follows from   \eqref{eq:supV} that
\begin{align*}
 \sup_{1\leq k\leq\etan}\left|\xen k-(\ton-k)\log n\right|
&\leq\sup_{1\leq k\leq2n/\log n}\left|S_k-k\log\left(\frac{2n}{\log n}\right)\right|
+O_p\left(\frac{n\log\log n}{\log n}\right)\\
&=O_p\left(\frac{n\log\log n}{\log n}\right).
\end{align*}
Finally using (\ref{eq:asser2step1}) and the strong Markov property for $\tilde X_k^{(n)}= \xen{k+\etan}$ , we deduce
\begin{equation}\label{eq:tau-eta}
 \ton-\etan
 =\tau^{(\tilde X_0^{(n)})}=\frac{\xen\etan}{\log \xen\etan}(1+o_p(1))
 =\frac{cn}{(\log n)^{1+\gamma}}(1+o_p(1)).
\end{equation}
Then, putting all the pieces together, we get
$$
 \sup_{1\leq k\leq \etan}\left|
\frac{\xen k}{\ton-k}-\log n\right|
\leq
\frac{ \sup_{1\leq k\leq \etan}\left|
\xen k-(\ton-k)\log n\right|}{\ton-\etan}
=O_p((\log n)^\gamma\log\log n),
$$
and since $\gamma<1$,
\begin{equation}
\label{eq:recu1}
 \sup_{1\leq k\leq \etan}\left|
\frac{\xen k}{\ton-k}-\log n\right|
=o_p(\log n).
\end{equation}

Next, we will prove (\ref{eq:siginfeta}), (\ref{eq:ratio succ}), (\ref{eq:asser2step1}), (\ref{eq:tau-eta}) and (\ref{eq:recu1})
for any $\gamma>0$.
We show that this claim holds for $\gamma\leq{p/2}$ for any $p \in \mathbb N$, using induction on $p$. The proof  for $p=1$ has just been done.

For the induction step suppose that the asymptotics in (\ref{eq:siginfeta}) to (\ref{eq:recu1}) hold for $\gamma\leq{p/2}$. For simplicity, we write
$\hetan=\etang{c,{p}/{2}}$. The idea is to use the strong Markov property at the stopping time $\hetan $  and apply the above results for $\gamma < 1$ to the Markov chain $\hat X_k^{(n)}= X_{k+\hetan}^{(n)}$ started at $\hat n= X_{\hetan}^{(n)}$ (instead of $n=\xen 0$). Define the family of stopping times
$$
\zetan=
\inf\left\{k, \xen k<\frac{\hat n}{(\log\hat n)^{2/3}}\right\}.
$$

Observe that $\zetan=\hetan+\eta_{1,\frac{2}{3}}^{(\hat n)}$.
Hence, using the strong Markov property  at the stopping time $\hetan $ and the behaviour in (\ref{eq:asser2step1}), with $\gamma=2/3$, we get
$$
\frac{(\log \xen\hetan)^{\frac{2}{3}}}{\xen\hetan}\xen\zetan
\xrightarrow[n\to\infty]{\P}1.
$$
Then, from  this asymptotic behaviour and the induction hypothesis taken in (\ref{eq:asser2step1}),
$$
\frac{(\log n)^{\frac{2}{3}+\frac{p}{2}}}{cn}\xen\zetan
\xrightarrow[n\to\infty]{\P}1.
$$
From this behaviour, from (\ref{eq:asser2step1}) and from
$$
\frac{(\log n)^\gamma}{cn}\xen{\etan} < 1 \le \frac{(\log n)^\gamma}{cn}\xen{\etan-1}
$$
we obtain for $p/2<\gamma\leq(p+1)/2$
\begin{equation}\label{eq:hetaetazeta}
\P\Big(\hetan< \etan\le \zetan\Big)\xrightarrow[n\to\infty]{} 1. 
\end{equation}

Now, on the event $\{\son>\hetan\}$, using the strong Markov property at $\hetan$ and (\ref{eq:siginfeta}) with the initial state $\xen\hetan$, we get
$$
\P\Big(\son\le \zetan\big|\son>\hetan\Big)\xrightarrow[n\to\infty]{}  0 .
$$
The induction hypothesis gives $\P(\son>\hetan)\to 1$, as $n$ goes to $\infty$.  These two facts together lead to
$$
\P\Big(\son>\etan\Big)\xrightarrow[n\to\infty]{}  1,
$$
for $\gamma\in (p/2,(p+1)/2]$.

From (\ref{eq:ratio succ}) and again the strong Markov property at $\hetan$, we get
$$
  \sup_{\hetan\leq k\leq\zetan}\left|\frac{\xen k}{\xen {k-1}}-1\right|
=o_p(1).
$$
From the above behaviour, (\ref{eq:hetaetazeta}) and  the induction hypothesis, we have
$$
  \sup_{1\leq k\leq\etan}\left|\frac{\xen k}{\xen {k-1}}-1\right|
=o_p(1),
$$
for $\gamma\in (p/2,(p+1)/2]$.
Again  the strong Markov property together with the above behaviour give us  for all $\gamma\in (p/2,(p+1)/2]$,
\[
 \ton-\etan =\frac{cn}{(\log n)^{1+\gamma}}(1+o_p(1)).
\]

From (\ref{eq:recu1}) and the strong Markov property, we get
$$
 \sup_{\hetan\leq k\leq \zetan}\left|
\frac{\xen k}{\ton-k}-\log \xen\hetan\right|
=o_p\Big(\log \xen\hetan\Big).
$$
We know from the induction hypothesis that $\log \xen\hetan\sim\log n$, as $n$ goes to $\infty$.
We then obtain, from  (\ref{eq:hetaetazeta}) and using again the induction hypothesis, 
$$
 \sup_{1\leq k\leq \etan}\left|
\frac{\xen k}{\ton-k}-\log n\right|
=o_p(\log n),
$$
for all $\gamma\in (p/2,(p+1)/2]$. Hence the induction is complete and the behaviour in
(\ref{eq:siginfeta}) to (\ref{eq:recu1}) hold for any $\gamma>0$.
\end{proof}

\begin{proof}[Proof of Proposition \ref{lem:grolem}] 
We first  recall the definition of $\thtng$ in (\ref{eq:defth}) and
define $\etang-=\etang{1-\varepsilon,\gamma}$ and $\etang+=\etang{1+\varepsilon,\gamma}$.
From (\ref{eq:tau-eta}), it is clear
$$
\P\Big(\etang+\leq\thtng\leq\etang-\Big)\xrightarrow[n\to\infty]{} 1,
$$
and from (\ref{eq:siginfeta}), we deduce
$$
\P\Big(\son>\etang-\Big)\xrightarrow[n\to\infty]{}  1.
$$
Thus the first asymptotic behaviour in  (\ref{eq:aser2}) holds. Also note
$$
\P\Big(\xen{\etang-}\leq\xen\thtng\leq\xen{\etang+}\Big)\xrightarrow[n\to\infty]{} 1,
$$
then the second asymptotic behaviour in (\ref{eq:aser2}) follows from (\ref{eq:asser2step1}).

From (\ref{eq:recu1}), we get
$$
 \sup_{1\leq k\leq\thtng}\left|
\frac{\xen k}{\ton-k}-\log n\right|
=o_p(\log n)
$$
which gives (\ref{eq:aser3}) for $\gamma'=\infty$.
Also
$$
 \sup_{1\leq k\leq\thtng}\left|
\frac{\thtn{\gamma'}-k}{\ton-k}-1\right|
=
 \sup_{1\leq k\leq\thtng}\left|
\frac{\ton-\thtn{\gamma'}}{\ton-k}\right|
\leq
\frac{\ton-\thtn{\gamma'}}{\ton-\thtng}
=
\frac{(\log n)^{\gamma}}{(\log n)^{\gamma'}}(1+o_p(1)).
$$
This give us (\ref{eq:aser3}). This completes the proof.
\end{proof}

\subsection{Proof of Theorem \ref{th:cvI}} 

We first define 
$$\tien=\sum_{k=1}^{\tau^{(n)}-1}\frac{Y_{k}^{(n)}}{\xen k},$$
which is obtained by replacing the exponential random variables $\mathbf{e}_k$'s by their mean and approximating the denominator. Similarly, we define 
$$\hien=\sum_{k=1}^{\tau^{(n)}-1}\frac{\E[Y_{k}^{(n)}|\xen{}]}{\xen k},$$
which is obtained by replacing the random variables $\yyen k$ by its conditional expectation.
This new formulation is of interest. 
Indeed, similar as in \cite{KSW13} it is possible to determine $\hien$ via a recursive formula.
Let $\zen k$ be the number of external branches after $k$ jumps, $k\geq1$, and we take conditional expectation to each $\zen k$ with respect to $\xen{}$ and  $\zen{k-1}$. Observe that $\zen{k-1}-\zen{k}$ is the number of external branches which participate to $k$-th coalescent event.
Hence, this random variable is distributed as an   hypergeometric r.v. with parameters
$\xen{k-1}$, $\zen{k-1}$ and $1+\uen k$.
It is then clear 
$$
\E\Big[\zen k\Big|\xen{},\zen{k-1}\Big]=
\zen{k-1}-\Big(1+\uen k\Big)\frac{\zen{k-1}}{\xen{k-1}},
$$
then $$
\E\Big[\zen k\Big|\xen{}\Big]
=\E\Big[\zen {k-1}\Big|\xen{}\Big]\frac{\xen k-1}{\xen {k-1}},
$$
and
\[
\frac{\E\Big[\zen k\Big|\xen{}\Big]}{\xen k}
=\prod_{i=1}^k\left(1-\frac{1}{\xen i}\right).
\]
Finally, since  $\yyen k=\xen k-\zen k$, it follows
\begin{equation}\label{eq:newdefhien}
  \hien=\sum_{k=1}^{\tau^{(n)}-1}\left(1-\prod_{i=1}^k\left(1-\frac{1}{\xen i}\right)\right).
\end{equation}
This last expression is a good way to understand the asymptotic behaviour of the total internal branch.

The following lemma provides the asymptotic behaviour of $ \hien$.
\begin{lem}\label{lemm:aincv}
As $n$ goes to $\infty$,
$$
\frac{(\log n)^2}{n}\hien \xrightarrow []{}1,
$$
in probability.
\end{lem}
\begin{proof}
Let $\varepsilon>0$ and take $\thtng$ as  in (\ref{eq:defth}). We also let $\thtn-=\lfloor \thtn{1-\varepsilon}\rfloor$ and $\thtn+=\lfloor \thtn{1+\varepsilon}\rfloor$ and consider $\hien$ as it is given in \eqref{eq:newdefhien}. We now split $\hien$ in two parts, as follows
$$
\hen 1=\sum_{k=1}^{\thtn+ -1}\left(1-\prod_{i=1}^k\left(1-\frac{1}{\xen i}\right)\right),
$$
and
$$
\hen 2=\sum_{k=\thtn+}^{\ton-1}\left(1-\prod_{i=1}^k\left(1-\frac{1}{\xen i}\right)\right).
$$
Note that
$$\hen2\leq\ton-\thtn+\le\frac{n}{(\log n)^{2+\varepsilon}}+1, $$
which implies that 
\[
\frac{(\log n)^2}{n} \hen2\xrightarrow[n\to\infty]{} 0,\qquad  \textrm{ almost surely.} 
\] 
Then it is enough to study the behaviour of $\hen1$. In order to do so, we first note
$$
\sum_{i=1}^{k}\frac{1}{\xen i}
-\sum_{j=2}^{k}\sum_{i=1}^{j-1}\frac{1}{\xen i\xen j}
\leq
1-\prod_{i=1}^k\left(1-\frac{1}{\xen i}\right)
\leq
\sum_{i=1}^{k}\frac{1}{\xen i}.
$$
(This can be viewed as two Bonferroni inequalities for independent events with entrance probabilities $1/\xen i$.) 

On the one hand, 
$$ \sum_{k=1}^{\thtn+ -1}
 \sum_{i=1}^{k}\frac{1}{\xen i}
= \sum_{i=1}^{\thtn+ -1}
\frac{\thtn+-i}{\xen i}
$$
  and thus
$$ \sum_{i=1}^{\thtn- -1}
\frac{\thtn+-i}{\xen i}\le \sum_{k=1}^{\thtn+ -1}
 \sum_{i=1}^{k}\frac{1}{\xen i}
\le\sum_{i=1}^{\thtn+ -1}
\frac{\ton-i}{\xen i}.
$$ 
From  (\ref{eq:aser3}), we get 
$$ \frac{1}{\log n}(\thtn--1)(1+o_p(1))
\le \sum_{k=1}^{\thtn+ -1}
 \sum_{i=1}^{k}\frac{1}{\xen i}
\le\frac{1}{\log n}(\thtn+-1)(1+o_p(1)).
$$ 
From  the fact that $\thtn-,\thtn+\sim\ton\sim n/\log n$, as $n\to\infty$, we deduce
  \[
   \sum_{k=1}^{\thtn+ -1}
 \sum_{i=1}^{k}\frac{1}{\xen i}
=\frac{n}{(\log n)^2}(1+o_p(1)).
\]
On the other hand by inverting the sums, we obtain
\[
\sum_{k=1}^{\thtn+ -1}\sum_{j=2}^{k}\sum_{i=1}^{j-1}\frac{1}{\xen i\xen j}=
\sum_{j=2}^{\thtn+ -1}\sum_{k=j}^{\thtn+ -1}\sum_{i=1}^{j-1}\frac{1}{\xen i\xen j}= \sum_{j=2}^{\thtn+ -1}
\frac{\thtn+-j}{\xen j}\sum_{i=1}^{j-1}\frac{1}{\xen i}.
\]
Using   (\ref{eq:aser3}), we obtain \[
\sum_{k=1}^{\thtn+ -1}\sum_{j=2}^{k}\sum_{i=1}^{j-1}\frac{1}{\xen i\xen j}\le \sum_{j=2}^{\thtn+ -1}
\frac{\ton-j}{\xen j}\sum_{i=1}^{j-1}\frac{1}{\xen i}\le\frac {1+o_p(1)}{\log n}\sum_{j=1}^{\thtn+ -1}\sum_{i=1}^{j}\frac{1}{\xen i}.
\]
and finally
 \[
\sum_{k=1}^{\thtn+ -1}\sum_{j=2}^{k}\sum_{i=1}^{j-1}\frac{1}{\xen i\xen j}\le\frac n{(\log n)^3}(1+o_p(1)).
\]
Putting all the pieces together give us 
 \[
 \frac{(\log n)^2}{n}\hen1\xrightarrow[n\to\infty]{\P}1,
\]
which ends the proof.
\end{proof}

In order to prove Theorem \ref{th:cvI}, we just need to control our approximation. This is the aim of the next two lemmas.
\begin{lem}
\label{lem:app1}
As $n$ goes to $\infty$, 
$$ \ien-\tien = O_P(\sqrt n).$$
\end{lem}
\begin{proof}
Recall that $\xen{}$ denotes the Markov chain $(\xen k,  k\geq 0)$. 
A simple computation gives us
\[
\Big(\ien-\tien\Big)=
\sum_{k=1}^{\tau^{(n)}-1}\yyen k
 \frac{\mathbf{e}_k-1 }{\xen k}
 +\sum_{k=1}^{\tau^{(n)}-1}\yyen k
 \frac{\mathbf{e}_k}{\xen k(\xen k-1)} .
\]
Conditionally on $\xen{},\yyen{}$, the random variables $\displaystyle
   \yyen k\frac {\mathbf{e}_k-1}{\xen k }$ are independent with zero mean. This implies
$$
 \E\left[\left(\sum_{k=1}^{\tau^{(n)}-1}\yyen k
 \frac{\mathbf{e}_k-1 }{\xen k}\right)^2 \bigg|\xen{},\yyen{}\right]
= 
\sum_{k=1}^{\tau^{(n)}-1} \left(\frac{\yyen k
}{{\xen k}}\right)^2
\leq\ton \leq n,
$$
where the inequality follows  from the fact that $\yyen k\leq \xen k$ a.s.
Chebychev's inequality  implies  
\[
 \sum_{k=1}^{\tau^{(n)}-1}\yyen k
 \frac{\mathbf{e}_k-1 }{\xen k} = O_P(\sqrt n).
\]

Again using  that $\yyen k\leq \xen k$ a.s., we get
$$
\sum_{k=0}^{\tau^{(n)}-1}\yyen k
 \frac{\mathbf{e}_k}{\xen k(\xen k-1)}
 \leq \sum_{k=0}^{\tau^{(n)}-1}
 \frac{\mathbf{e}_k}{(\xen k-1)}
 \leq \sum_{k=0}^{{n}-1}
 \frac{\mathbf{e}_k}{( k-1)}.
$$
It is a classical result (used also for the total length of Kingman coalescent) that 
$$
\frac{1}{\log n}
\sum_{k=0}^{{n}-1}
 \frac{\mathbf{e}_k}{( k-1)}\xrightarrow[n\to\infty]{\P}1,
 $$
 which implies that 
 \[ \sum_{k=0}^{\tau^{(n)}-1}\yyen k
 \frac{\mathbf{e}_k}{\xen k(\xen k-1)} = O_P(\log n).
 \]
 This completes the proof.
\end{proof}

\begin{lem}
\label{lem:app2} As $n$ goes to $\infty$
$$ \tien-\hien = O_P(\sqrt n).$$

\end{lem}

\begin{proof}
We proceed similar as in \cite{KSW13}. Recall that $\zen k$ is the number of external branches after $k$ coalescing events. Since $\yyen k= \xen k-\zen k$, 
\[\tien-\hien = -\sum_{k=1}^{\tau^{(n)}-1}\frac{\zen k-\mathbb E[\zen k|\xen{}]}{\xen k}
\]
Also recall that $\zen k-\zen {k-1}$ has a conditional hypergeometric distribution, given $\xen {}, \zen {k-1}$. Therefore
\[ \zen k= \zen {k-1} - (\uen k+1)\frac{\zen {k-1}}{\xen {k-1}} - H^{(n)}_k = \zen {k-1}\frac{\xen k-1}{\xen{k-1}}- H^{(n)}_k , \]
where $H_{k}^{(n)}$ denotes a random variable with conditional hypergeometric distribution with parameters
$\xen{k-1}$, $\zen{k-1}$ and $1+\uen k$ as above, centered at its (conditional) expectation. For
\[ D_k^{(n)} = \zen k- \mathbb E [\zen k | \xen{}] \]
it follows
\[ D_k^{(n)}=  D_{k-1}^{(n)}\frac{\xen k-1}{\xen{k-1}}- H^{(n)}_k\ . \]
Iterating this linear recursion we obtain because of $D_0^{(n)}=0$
\[   \frac{D_k^{(n)}}{\xen k} = -\sum_{j=1}^k \frac{ H^{(n)}_j}{\xen j} \prod_{i=j+1}^k \Big( 1- \frac 1{\xen i}\Big) \]
and consequently
\[ \tien-\hien =  \sum_{k=1}^{\tau^{(n)}-1} \sum_{j=1}^k \frac{ H^{(n)}_j}{\xen j} \prod_{i=j+1}^k \Big( 1- \frac 1{\xen i}\Big)=  \sum_{j=1}^{\tau^{(n)}-1} \frac{ H^{(n)}_j}{\xen j} \sum_{k=j}^{\tau^{(n)}-1}\prod_{i=j+1}^k \Big( 1- \frac 1{\xen i}\Big) .\]
Now, since the $H_k^{(n)}$ are centered hypergeometric variables, they are uncorrelated, given $\xen {}$. Also from the formula for the variance of a hypergeometric distribution
\[ \mathbb E[(H_j^{(n)})^2 | \xen{}, Z_{j-1}^{(n)}] \le (\uen j+1)\frac{\zen {j-1}}{\xen {j-1}} \]
thus, since $\zen {j-1} \le \xen {j-1}$ a.s.
\[ \mathbb E[(H_j^{(n)})^2 | \xen{}]  \le (\uen j+1)  .\]
Putting everything together we obtain
\[ \mathbb E[ (\hien - \tien)^2 | \xen {}] \le  \sum_{j=1}^{\tau^{(n)}-1} \frac{\uen j+1}{(\xen j)^2}\Big(\sum_{k=j}^{\tau^{(n)}-1}\prod_{i=j+1}^k \Big( 1- \frac 1{\xen i}\Big)\Big)^2.\]
The product can be estimated by 1, thus
\[ \mathbb E[ (\hien - \tien)^2 | \xen {}] \le  \sum_{j=1}^{\tau^{(n)}-1} \frac{\uen j+1}{(\xen j)^2} (\tau^{(n)} - j)^2.\]
By means of $\tau^{(n)} - j \le \xen j$
\[  \mathbb E[ (\hien - \tien)^2 | \xen {}] \le  \sum_{j=1}^{\tau^{(n)}-1} (\uen j+1) \le n+ \tau_n \le 2n . \]
Now an application of Chebychev's inequality gives the claim.
\end{proof}

\section{Application to population genetics}\label{sec:mut}

Let us now suppose that mutations occur along genealogical trees according to a Poisson process of intensity $\mu$. 
We write by  $\men{}$ for  the total number of mutations in the Bolthausen-Sznitman $n$-coalescent. The Poissonian representation implies that, conditionally on $\len$, $\men{}$ is distributed as a Poisson r.v. with parameter $\mu\len$. Mutations can be divided as external and internal according to the type of the branches where they appear and we denote them  by $\men E$ and $\men I$, respectively.

\begin{prop}
 As  $n$ goes to $\infty$,
$$
\frac{(\log n)^2}{n}\men I\to \mu,
$$
in probability and
$$
\frac{(\log n)^2}{n}\men E-\mu\log n-\mu\log\log n\to  \mu(Z-1),
$$
in distribution.
\end{prop}

\begin{proof} Let $N=(N_t,t\ge 0)$ be a Poisson process with parameter $\mu$. We first note that conditionally on $\ien$, $M^{(n)}_I$  has the same distribution as $N_{I^{(n)}}$. This implies
 $$
 \E\Big[\men{I}\Big]=\E\Big[\E\Big[\men{I}|\ien\Big]\Big]=\mu\E\Big[\ien\Big]\ .
 $$
Since $\ien\to\infty$ a.s.,  thanks  Theorem \ref{th:cvI}, we deduce that $N_{\ien}/\ien\to\mu$ in probability  and
$$
\frac{\men I}{\E\Big[\men I\Big]}\stackrel{d}{=}
\frac{N_{\ien}}{\mu\ien}\frac{\ien }{\E\Big[\ien \Big]}\xrightarrow[n\to\infty]{\P}1,
$$
thanks to Theorem \ref{th:cvI}.
Therefore, the first result follows from
$ \E[\men{I}]=\mu\E[\ien]\sim \mu n/(\log n)^2$, as $n\to\infty$.

To get the second part of this proposition, we just need to observe that $\men{}=\men I+\men E$ satisfies (see Corollary 6.2 of \cite{DIMR07})
$$
\frac{(\log n)^2}{n}\men{}-\mu\log n-\mu\log\log n
\xrightarrow[n\to\infty]{}\mu Z,
$$
in distribution.
\end{proof}

{\bf Acknowledgement.}
G.K. acknowledges support by the DFG  priority program
SPP-1590 'Probabilistic structures in Evolution.
J.C.P.  acknowledges support by CONACYT,  No. 128896. 


\end{document}